\documentclass[a4paper,11pt,oneside,reqno]{amsart}
\usepackage{amsmath,amssymb,amsthm,mathtools}
\usepackage[pagewise]{lineno}
\mathtoolsset{showonlyrefs}
\usepackage{bm}

\usepackage{subfigure}
\usepackage{verbatim}
\usepackage{wrapfig}
\usepackage{ascmac}
\usepackage{cases}
\usepackage{amsthm}
\usepackage{autobreak}

\usepackage{comment}
\usepackage {cancel}
\usepackage[normalem]{ulem}

  \makeatletter
    
    \@addtoreset{equation}{section}
  \makeatother
\newtheorem{theorem}{Theorem}[section]

\newtheorem{lemma}[theorem]{Lemma}

\newcounter{c}
\newcounter{d}
\newcounter{b}

\newcommand{\R}{\mathbb{R}} 

\newcommand{\e}{\varepsilon} 
\newcommand{\supp}{\mathrm{supp}} 

\begin{document}
\title{Vanishing interfaces in an asymmetric fast reaction limit}
\author{Yuki Tsukamoto}
\date{}

\address{Tokyo University of Science,
	162-8601, Tokyo, Japan}
\email{tsukamoto-yuki@rs.tus.ac.jp} 

\begin{abstract}
We study the fast reaction limit for a two-component reaction-diffusion system
with asymmetric reaction terms, where only one component diffuses.
For nonnegative and mutually segregated initial data, we prove that the initial
interface vanishes instantaneously. More precisely, the diffusive component
converges uniformly to the solution of the heat equation, while the non-diffusive
component vanishes away from the initial time. The proof is based on explicit
barriers and a comparison argument, and applies under both Dirichlet and
Neumann boundary conditions.
\end{abstract}
\maketitle

\section{Introduction}

Reaction-diffusion systems provide a fundamental mathematical framework for
describing the interplay between spatial diffusion and local reaction. Such
systems arise in a wide variety of contexts, including chemical reactions,
population dynamics, materials science, and transport phenomena in porous media.
A typical two-component system is written in the form
\[
    \partial_t u = d_1\Delta u + f_1(u,v),\qquad
    \partial_t v = d_2\Delta v + f_2(u,v),
\]
where the diffusion terms describe spatial spreading, while the reaction terms
represent local interactions between the components.

In many applications, the reaction process takes place on a much faster time
scale than diffusion. To describe such regimes, one introduces a large
reaction-rate parameter \(k>0\) and studies the singular limit as
\(k\to\infty\). This singular limit, known as the fast reaction limit, has been
studied extensively in the theory of reaction-diffusion systems. 
Evans proved one of the early convergence results for a chemical
diffusion-reaction system \cite{evans1980convergence}. 
Hilhorst, van der Hout and Peletier studied
irreversible fast reactions leading to Stefan-type limiting problems
\cite{hilhorstFastReactionLimit1996}. They also treated more general monotone
reaction terms \cite{hilhorst1997diffusion}. Eymard, Hilhorst, van der Hout
and Peletier studied a reaction-diffusion approximation of a one-phase Stefan
problem \cite{eymard2001reaction}, while Bothe and Hilhorst analyzed systems
with fast reversible reactions \cite{bothe2003reaction}.

Fast reaction limits are also closely related to spatial segregation limits in
competition-diffusion systems. Dancer, Hilhorst, Mimura and Peletier studied
spatial segregation as the competition rate tends to infinity
\cite{dancer1999spatial}, and Crooks, Dancer, Hilhorst, Mimura and Ninomiya
analyzed the corresponding problem under Dirichlet boundary conditions
\cite{crooks2004spatial}. Bouillard, Eymard, Henry, Herbin and Hilhorst studied
a fast precipitation and dissolution reaction in a porous medium, providing an
example in the context of reactive transport \cite{bouillard2009fast}.

Many of the classical examples above belong to the class of balanced systems,
in which the two components are consumed by the same fast reaction term, or by
proportional fast reaction terms. A typical form is
\[
    \partial_t u_k = d_1\Delta u_k - kF(u_k,v_k),\qquad
    \partial_t v_k = d_2\Delta v_k - kF(u_k,v_k).
\]
In such systems, the fast reaction suppresses the region where both components
are simultaneously positive, and the limit is often described by a one- or
two-phase Stefan-type free boundary problem.

The theory of fast reaction limits has recently been extended in several directions. Stephan proved
EDP-convergence for a linear reversible reaction-diffusion system using its
gradient-flow structure \cite{stephan2021edp}. Perthame and Skrzeczkowski
treated nonmonotone reaction functions and described the limit by Young
measures \cite{perthame2023fast}. Skrzeczkowski studied the connection with
forward-backward diffusion by a Radon--Nikodym type approach
\cite{skrzeczkowski2022}. Crooks and Du considered nonlinear diffusion
\cite{crooks2024fast}, and Murakawa studied fast reaction limits within a
general approximation framework \cite{murakawa2021fast}. These works illustrate
that, even in balanced or closely related settings, the limiting behavior
depends strongly on the reaction and diffusion structures.

In contrast, when the fast reaction terms in the two equations are not
proportional, the problem becomes substantially different. Such systems are
often called unbalanced systems or asymmetric fast reaction systems. The
cancellation structures available in balanced systems are no longer directly
available, and the limiting interface dynamics may differ essentially from the
Stefan-type behavior of balanced systems. Conti, Terracini and Verzini
\cite{conti2005asymptotic} and Hilhorst, Iida, Mimura and Ninomiya
\cite{hilhorst2008relative} studied steady multi-component
competition-diffusion systems related to such unbalanced structures.
Time-dependent unbalanced systems were studied by Iida, Monobe, Murakawa and
Ninomiya \cite{Iida2017} and by Hayashi \cite{hayashi2021spatial}.

A representative model studied by Iida, Monobe, Murakawa and Ninomiya
\cite{Iida2017} is
\[
\begin{cases}
    \partial_t u_k = \Delta u_k - k u_k^{m_1}v_k^{m_2},\\
    \partial_t v_k = -k u_k^{m_3}v_k^{m_4},
\end{cases}
\]
where only the \(u_k\)-component diffuses. They considered nonnegative and
mutually segregated initial data, \(u_0v_0\equiv0\), so that an initial
interface is formed between the regions occupied by \(u_0\) and \(v_0\).
Their analysis focused on representative one-parameter families obtained by
fixing three of the four exponents \(m_1,m_2,m_3,m_4\) equal to one and varying
the remaining exponent.

More precisely, in the case
\((m_1,m_2,m_3,m_4)=(m_1,1,1,1)\), they proved that the interface vanishes
instantaneously when \(m_1>3\). In the case \((1,m_2,1,1)\) with \(m_2>1\),
the interface moves with finite speed and is described by a one-phase Stefan
problem. In the case \((1,1,m_3,1)\) with \(m_3>1\), the interface remains
stationary. In the case \((1,1,1,m_4)\) with \(1\le m_4<2\), the interface
again moves with finite speed. Thus the asymmetry of the reaction terms gives
rise to vanishing, moving, and stationary interfaces, depending on which
reaction exponent is varied.

In the classification of \cite{Iida2017}, the ranges \(1<m_1\le3\) in the
vanishing case and \(m_4\ge2\) in the fourth case remained open. These remaining
ranges were partially resolved in the author's recent works: the case
\(m_4\ge2\) was treated in \cite{tsukamotoStationary}, and instantaneous
interface disappearance was proved for \(2<m_1\) in \cite{tsukamoto2025}.
Thus, in the vanishing case, the range \(1<m_1\le2\) remained open.

The present paper treats this remaining range. More precisely, we prove that,
for every \(m>1\), the initial interface for
\[
\begin{cases}
    \partial_t u_k = \Delta u_k - k u_k^m v_k,\\
    \partial_t v_k = -k u_k v_k,
\end{cases}
\]
vanishes instantaneously under both Dirichlet and Neumann boundary conditions.
Consequently, together with \cite{Iida2017}, the author's previous work
\cite{tsukamoto2025}, and the author's related work on stationary interfaces
\cite{tsukamotoStationary}, the present result completes the classification of
the representative one-parameter exponent regimes in which three of the four
exponents are fixed equal to one.

The main results of this paper are stated as follows.
\begin{theorem}\label{thm:Dirichlet}
Let $\Omega\subset\R^n$ be a bounded domain with $C^2$ boundary 
$\partial\Omega$, and let $T>0$ and $m>1$. 
For each $k>0$, let $(u_k,v_k)$ denote the classical solution of
\begin{align*}
\begin{cases}
\partial_t u_k = \Delta u_k - k\,u_k^{m} v_k 
    & \text{in } Q_T:=\Omega\times(0,T),\\[2mm]
\partial_t v_k = - k\,u_k v_k 
    & \text{in } Q_T,\\[2mm]
u_k = 0 
    & \text{on } S_T:=\partial\Omega\times(0,T),\\[2mm]
u_k(\cdot,0)=u_0,\quad v_k(\cdot,0)=v_0 
    & \text{in } \Omega,
\end{cases}
\end{align*}
where 
\[
u_0\in C(\overline{\Omega}),\quad 
v_0\in L^\infty(\Omega),\quad
u_0\ge0,\ v_0\ge0,\quad 
u_0 v_0 \equiv 0,\quad 
u_0\not\equiv0,\ v_0\not\equiv0,
\]
and $u_0=0$ on $\partial\Omega$.

Then $u_k$ converges uniformly in $\overline{Q_T}$ to a function
$u_\infty\in C(\overline{Q_T})\cap C^\infty (Q_T)$, and $u_\infty$ is the unique classical 
solution of
\begin{align*}
\begin{cases}
\partial_t u_\infty = \Delta u_\infty & \text{in } Q_T,\\[2mm]
u_\infty = 0 & \text{on } S_T,\\[2mm]
u_\infty(\cdot,0)=u_0 & \text{in } \Omega.
\end{cases}
\end{align*}

Moreover, for any $\tau\in(0,T)$ and any domain $\Omega_0\Subset\Omega$, 
\[
v_k \to 0 \quad \text{uniformly in } \overline{\Omega_0}\times[\tau,T].
\]
\end{theorem}

\begin{theorem}\label{thm:Neumann}
Let $\Omega\subset\R^n$ be a bounded domain with $C^2$ boundary 
$\partial\Omega$, and let $T>0$ and $m>1$.
For each $k>0$, let $(u_k,v_k)$ be the classical solution of
\begin{align*}
\begin{cases}
\partial_t u_k = \Delta u_k - k\,u_k^{m} v_k 
    & \text{in } Q_T:=\Omega\times(0,T),\\[2mm]
\partial_t v_k = - k\,u_k v_k 
    & \text{in } Q_T,\\[2mm]
\partial_\nu u_k = 0 
    & \text{on } S_T:=\partial\Omega\times(0,T),\\[2mm]
u_k(\cdot,0)=u_0,\quad v_k(\cdot,0)=v_0 
    & \text{in } \Omega,
\end{cases}
\end{align*}
where
\[
u_0\in C^1(\overline{\Omega}),\quad 
v_0\in L^\infty(\Omega),\quad
u_0\ge0,\ v_0\ge0,\quad 
u_0 v_0 \equiv 0,\quad 
u_0\not\equiv0,\ v_0\not\equiv0,
\]
and $\partial_\nu u_0 = 0$ on $\partial\Omega$.

Then $u_k$ converges uniformly in $\overline{Q_T}$ to a function 
$u_\infty\in C(\overline{Q_T})\cap C^\infty (Q_T)$, which solves
\begin{align}\label{heat-Neumann}
\begin{cases}
\partial_t u_\infty = \Delta u_\infty & \text{in } Q_T,\\[2mm]
\partial_\nu u_\infty = 0 & \text{on } S_T,\\[2mm]
u_\infty(\cdot,0)=u_0 & \text{in } \Omega.
\end{cases}
\end{align}

Furthermore, for every $\tau\in(0,T)$,
\[
v_k \to 0 \quad \text{uniformly in } \overline{\Omega}\times[\tau,T].
\]
\end{theorem}

We briefly explain the idea of the proof and the difference from
\cite{tsukamoto2025}. In \cite{tsukamoto2025}, the key step was to obtain a
lower bound for \(u_k\) from a suitably chosen auxiliary heat equation. This
argument was effective for \(m>2\), but it relied on estimates that do not
extend to the range \(1<m\le2\).

The present paper uses a different barrier construction. We first construct a
one-dimensional barrier for an auxiliary Dirichlet problem associated with the
fast reaction limit. This barrier is then lifted to higher-dimensional annular
regions as a radial function. By applying such annular barriers successively, we
propagate the positivity of \(u_k\) from the initial positive region to
arbitrary compact subsets of the domain. This
gives, for every \(\tau>0\), a positive lower bound independent of \(k\) for
\[
    \int_0^t u_k(x,s)\,ds
\]
on the relevant space-time regions. Since
\[
    v_k(x,t)
    =
    v_0(x)\exp\left(-k\int_0^t u_k(x,s)\,ds\right),
\]
the non-diffusive component \(v_k\) then converges to zero exponentially fast
away from the initial time. Once this estimate is obtained, the convergence of
\(u_k\) follows by comparison with the corresponding heat equation.

The paper is organized as follows. In Section~2, we establish a comparison
principle adapted to the asymmetric structure of the system. We also construct
the auxiliary barriers used in the proof: first a one-dimensional barrier for a
Dirichlet fast reaction problem, and then its radial extension to annular
subsolutions in higher dimensions. In Section~3, these barriers are applied
successively to propagate the positivity of \(u_k\) through the domain. This
yields the disappearance of \(v_k\) away from the initial time, and the proofs
of Theorems~\ref{thm:Dirichlet} and~\ref{thm:Neumann} are then completed by
comparison with the corresponding heat equation.

	\section{Comparison principle and auxiliary functions}

\begin{theorem} \label{theo:comparison principle}
Let $A, B \subset \R^n $ be bounded domains with piecewise $C^2$ boundaries.
Let $u, \tilde{u} \in C([0,T]; C(\overline{A})) \cap C^1((0,T]; C(\overline{A})) \cap C((0,T]; W^{2,p}(A))$
		and $v, \tilde{v} \in C^1([0,T]; L^\infty(A))$  be nonnegative functions. 
Assume they satisfy the inequalities
		\begin{equation} \label{eq:comparison}
			\begin{aligned}
				\partial_t u &\geq \Delta u - k u^m v, \quad
				\partial_t \tilde{u} \leq \Delta \tilde{u} - k \tilde{u}^m \tilde{v} &&\text{in } A \times (0,T), \\
				\partial_t v &\leq -k u v, \quad
				\partial_t \tilde{v} \geq -k \tilde{u} \tilde{v} &&\text{in } A \times (0,T), \\
				\partial_\nu u &\ge \partial_\nu \tilde{u} &&\text{on } ((\partial A) \setminus B) \times (0,T), \\
				 u &\ge  \tilde{u} &&\text{on } ((\partial A) \cap B) \times (0,T), \\
				u(x,0) &\geq \tilde{u}(x,0), \quad v(x,0) \leq \tilde{v}(x,0) &&\text{for } x \in A.
			\end{aligned}
			\end{equation}
			Then
			\begin{align*}
				u \geq \tilde{u} , v \leq \tilde{v} \quad  \text{in } A \times (0,T).
			\end{align*}
		\end{theorem}
		\begin{proof}
			For convenience, we set $A_T:= A \times (0,T)$. Define $U=\tilde u - u$ and $V = v - \tilde v$, and denote 
$U_+=\max\{U,0\}$ and $V_+=\max\{V,0\}$. 
We will verify that $U_+$ and $V_+$ vanish identically in $A_T$.
From the inequalities in \eqref{eq:comparison}, 
for any nonnegative test functions 
$\phi,\psi\in C^\infty (\overline{A_T})$ we obtain
\begin{align} \label{eq com2}
	\begin{split}
		\int_{A_T}
\bigl(
 &-U\,\partial_t\phi 
 + \nabla U\cdot\nabla\phi 
 + k(\tilde u^m \tilde v - u^m v)\phi
 \bigr)\,dx\,dt\\
&- \int_{((\partial A) \cap B)  \times (0,T)} \partial_\nu U \phi \,dx \,dt + \int_{A} U(x,T)\phi(x,T)\,dx 
\le 0, \\[2mm]
\int_{A_T}
\bigl(
 &-V \,\partial_t\psi
 + k( uv - \tilde u\tilde v ) \psi
 \bigr)\,dx\,dt
+ \int_{A} V(x,T)\psi(x,T)\,dx
\le 0.
	\end{split}
\end{align}
We now take $\phi = U_+ e^{-Lt}$ and $\psi = V_+ e^{-Lt}$ with a constant $L>0$ 
to be chosen sufficiently large later.  
Since these functions are not smooth, we approximate them by smooth test 
functions and pass to the limit.  
Applying \eqref{eq com2}, we obtain
\begin{align} \label{eq:com 3}
				\begin{split}
					0 \geq& \int_{A_T} e^{-Lt} \left( LU_+^2 -\frac{1}{2}\partial_t (U_+^2) + |\nabla U_+|^2  + k ( \tilde{u}^m \tilde{v}-u^mv) U_+ \right)  \,dx \,dt
					\\ &+ \int_{A} U_+^2(x,T) e^{-LT}  \, dx \\
					\geq&\int_{A_T} e^{-Lt} \left( \frac{LU_+^2}{2}   + k ( \tilde{u}^m \tilde{v}-u^mv) U_+  \right)  \,dx \,dt,\\
					0  \geq& \int_{A_T} e^{-Lt} \left( \frac{LV_+^2}{2} + k (u v-\tilde{u} \tilde{v} ) V_+  \right)  \,dx \,dt 
				\end{split}
			\end{align}
Next we estimate the reaction terms. 
From Young’s inequality, we obtain
\begin{align}\label{eq:com 4}
	\begin{split}
	(\tilde u^m \tilde v - u^m v) U_+ &\ge \tilde u^m ( \tilde v-v)U_+ +v(\tilde u^m- u^m)U_+\\
	&\ge -\frac12 \| \tilde u \|^m_{L^{\infty}}(U_+^2 + V_+^2),\\
(uv - \tilde u \tilde v) V_+ &\ge u(v - \tilde v)V_+ + \tilde v (u - \tilde u)V_+\\
	&\ge -\frac12 \|\tilde v \|_{L^{\infty}}(U_+^2 + V_+^2).
	\end{split}
\end{align}
We now choose  $L=2k(\| \tilde u \|^m_{L^{\infty}}+\|\tilde v \|_{L^{\infty}})$. Then, combining \eqref{eq:com 3} with \eqref{eq:com 4}, we obtain  
\begin{align*}
				0\geq& \int_{A_T}  \frac{Le^{-Lt}}{2}(U_+^2+V_+^2) -  
				e^{-Lt}\frac{k}{2}\left(\| \tilde u \|^m_{L^{\infty}}+\|\tilde v \|_{L^{\infty}}\right)(U_+^2+V_+^2)   \,dx\, dt.\\
				=& \int_{A_T} \frac{Le^{-Lt}}{4}(U_+^2+V_+^2)  \,dx\, dt.
			\end{align*}
	This implies that $U_+ = 0$ and $V_+ = 0$ in $A_T$.
			Thus, we have $u \geq \tilde{u}$ and $v \leq \tilde{v}$ in $A_T$.
		\end{proof}
In the next step, we construct a one-dimensional interpolation function,
which will play a key role in the analysis below.  
To this end, we first introduce a self-similar solution.
Let $U_0, V_0 > 0$ be arbitrary constants.  
We choose a constant $\iota= \iota(U_0,V_0) > 0$ such that
\begin{align} \label{iota}
	2 \frac{U_0}{V_0} = \iota \int_{0}^{\iota} e^{\frac{\iota^2-s^2}{4}}\, ds .
\end{align}
With this choice of $\iota$, we define the profile
\[
f^{(U_0,V_0)}(\eta)
=
\begin{cases}
\displaystyle
U_0 \left(
1 - 
\frac{\int_{0}^{\eta} e^{-s^2/4}\, ds}
     {\int_{0}^{\iota} e^{-s^2/4}\, ds}
\right)
& \text{if } \eta \le \iota, \\
0
& \text{if } \eta > \iota .
\end{cases}
\]
We then introduce the function $U(x,t) := f^{(U_0,V_0)} \left( x/{\sqrt{t}} \right)$.
One checks that $U$ satisfies
\begin{align} \label{limit function}
  \begin{cases}
    U_t = U_{xx}
      & \text{in } \{(x,t)\in (0,\infty)\times(0,T):\, x < \iota\sqrt{t}\}, \\[2mm]
    U(0,t) = U_0
      & \text{for } t\in(0,T), \\[2mm]
    U(\iota\sqrt{t},\, t) = 0
      & \text{for } t\in(0,T), \\[2mm]
    U(x,0) = 0
      & \text{for } x>0 .
  \end{cases}
\end{align}
On the half-line, it is known from the earlier work \cite{hilhorstFastReactionLimit1996} that
solutions of fast reaction systems converge, as $k\to\infty$,
to the self-similar profile \eqref{limit function}.  
Motivated by this result, we show in the next lemma that an analogous
convergence also holds under the Dirichlet boundary condition.
\begin{lemma}\label{lem:dirichlet}
Let $U_0, V_0, d_1, d_2 > 0$ be arbitrary constants with $d_1 < d_2$, 
and let $\iota>0$ be chosen so as to satisfy \eqref{iota}. 
Assume that $T \ge d_2^{2}\iota^{-2}$.
For each $k>0$, there exists a classical solution $(u_k, v_k)$ of
\begin{equation}\label{eq:lem2.2}
\begin{aligned}
  \partial_t u_k &= \partial_{xx} u_k - k u_k v_k
    && \text{in } (d_1,d_2)\times(0,T),\\
  \partial_t v_k &= - k u_k v_k
    && \text{in } (d_1,d_2)\times(0,T).
\end{aligned}
\end{equation}
subject to the boundary and initial conditions
\begin{align*}
  u_k(d_1,t) &\le  U_0 && \text{for } t\in(0,T),\\
  u_k(d_2,t) &= 0   && \text{for } t\in(0,T),\\
  u_k(x,0) &= 0     && \text{for } x\in(d_1,d_2),\\
  v_k(x,0) &= V_0   && \text{for } x\in(d_1,d_2),
\end{align*}
and satisfying the monotonicity and boundedness conditions
\begin{align*}
  &\partial_x u_k(x,t) \le 0, \qquad
    \partial_t u_k(x,t) \ge 0
    && \text{in } (x,t)\in(d_1,d_2)\times(0,T),\\
  &0 \le u_k(x,t) \le U_0
    && \text{in } (x,t)\in[d_1,d_2]\times[0,T].
\end{align*}
Then, as $k\to\infty$, the sequence $\{u_k\}$ converges to $U$ uniformly in $[d_1,d_2]\times[0,d_2^{2}\iota^{-2}]$,
where $U$ is the function characterized by \eqref{limit function}.
\end{lemma}
\begin{proof}
	We first recall the convergence result for the corresponding problem
posed on the half-line. For each $k>0$, let $(\tilde u_k,\tilde v_k)$
denote the classical solution of
\begin{align*}
  \partial_t \tilde u_k &= \partial_{xx} \tilde u_k - k \tilde u_k \tilde v_k
  && \text{in } (0,\infty)\times(0,T),\\
  \partial_t \tilde v_k &= - k \tilde u_k \tilde v_k
  && \text{in } (0,\infty)\times(0,T),
\end{align*}
subject to the boundary and initial conditions
\begin{align*}
  \tilde u_k(0,t) &= U_0 && \text{for } t \in (0,T),\\
  \tilde u_k(x,0) &= 0   && \text{for } x \in (0,\infty),\\
  \tilde v_k(x,0) &= V_0 && \text{for } x \in (0,\infty).
\end{align*}
By combining \cite[Theorem~4.1]{hilhorstFastReactionLimit1996} and \cite[Theorem~3.16]{crooksSelfsimilarFastreactionLimits2016}, the sequence
$\{\tilde u_k\}_{k>0}$ converges to the self-similar profile $U$
defined in \eqref{limit function} locally uniformly in
$([0,\infty)\times (0,T]) \setminus \{(0,0)\}$ as $k\to\infty$.

Next, for each $k>0$ we construct a solution $(u_k,v_k)$ on the interval $(d_1,d_2)$.
At $x=d_1$ we prescribe the boundary value
\[
    u_k(d_1,t)=\tilde u_k(d_1,t) \qquad (0<t< T),
\]
and at $x=d_2$ we impose the homogeneous Dirichlet condition together with the initial data
\[
    u_k(d_2,t)=0, \qquad
    u_k(x,0)=0, \qquad v_k(x,0)=V_0.
\]
By \cite{eymard2001reaction}, a solution to \eqref{eq:lem2.2} exists under these boundary and initial conditions,
and the usual regularity estimates ensure that $u_k$ and $v_k$ are smooth in the prescribed domain.

Moreover, by \cite[Theorem~1.1]{eymard2001reaction}, the sequence $\{u_k\}$
converges in $L^2\bigl((d_1,d_2)\times (0,T)\bigr)$ to some limit
$\widehat U$. Since the initial and boundary data of $U$ coincide with those of
$u_k$, the uniqueness of the limit
problem implies that $\widehat U = U$ in $[d_1,d_2]\times[0,d_2^{2}\iota^{-2}]$.

Since the comparison principle yields $\tilde u_k(d_1,t)\le U_0$, we also have $u_k(d_1,t)\le U_0$.
Moreover, following the argument in the proof of \cite[Lemma~4.2]{hilhorstFastReactionLimit1996},
we differentiate the equation for $u_k$ with respect to both $x$ and $t$, and apply the
maximum principle to the negative parts of $\partial_x u_k$ and $\partial_t u_k$. 
This yields
\[
    \partial_x u_k(x,t)\le 0, \qquad \partial_t u_k(x,t)\ge 0
    \qquad\text{in } (d_1,d_2)\times(0,T),
\]
and hence the required monotonicity condition is satisfied.

By Theorem~\ref{theo:comparison principle} we further have
\[
    0\le u_k(x,t)\le \tilde u_k(x,t) \le U_0
    \qquad\text{for all } (x,t)\in[d_1,d_2]\times[0,T].
\]
To prove uniform convergence, let $\e >0$ be arbitrary and define
\begin{align*}
	& A:=\bigl\{(x,t)\in(d_1,d_2)\times(0,d_2^{2}\iota^{-2}) :
        U(x,t)<\e /2 \bigr\},\\
	&			A^c:=((d_1,d_2)\times(0,d_2^{2}\iota^{-2}))\setminus A.
\end{align*}
Since $\tilde u_k\to U$ locally uniformly in
$([0,\infty)\times[0,d_2^{2}\iota^{-2}])\setminus\{(0,0)\}$,
there exists $K\in\mathbb{N}$ such that for all $k>K$,
\[
    \sup_{(x,t)\in (d_1,d_2)\times(0,d_2^{2}\iota^{-2})} |\tilde u_k(x,t)-U(x,t)| < \e /2.
\]
Using $0\le u_k\le \tilde u_k$, we then obtain for all $k>K$,
\[
    \sup_{(x,t)\in A} |u_k(x,t)-U(x,t)|
        \le \sup_{(x,t)\in A} |\tilde u_k(x,t)-U(x,t)|
            + \sup_{(x,t)\in A} U(x,t)
        < \e.
\]

Moreover, by the proof of Lemma~4.3 in~\cite{eymard2001reaction}, the family $z_k(x,t):=\int_0^t u_k(x,s)\,ds$
is uniformly bounded in $W^{2,1}_2\bigl((d_1,d_2)\times(0,d_2^{2}\iota^{-2})\bigr)$,
independently of $k$.
Since the space dimension is one, Morrey's inequality implies that
\[
    z_k \to z_\infty := \int_0^t U(x,s)\,ds \quad \text{uniformly on }
    [d_1,d_2]\times[0,d_2^{2}\iota^{-2}].
\]
Since $U\ge \e/2$ on $A^c$, the continuity of $U$ implies that
\[
    \inf_{(x,t)\in A^c} z_\infty \ge 2\delta
\]
for some constant $\delta>0$. The uniform convergence $z_k\to z_\infty$ then gives
\[
    \inf_{(x,t)\in A^c} z_k(x,t) \ge \delta
    \qquad\text{for all sufficiently large } k.
\]

For $(x,t)\in A^c$ we therefore have
\[
    k u_k(x,t) v_k(x,t)
      = k u_k V_0 e^{-k z_k}
      \le U_0 V_0\, k e^{-k\delta}.
\]
Since $k e^{-k\delta}$ is bounded on $(0,\infty)$, the reaction term
$k u_k v_k$ is uniformly bounded on $A^c$ for all large $k$.

Hence, applying \cite{ladyzhenskaya1968} on $A^c$, we deduce that
$u_k \to U$ uniformly on $A^c$. Combining this with the estimate on $A$ yields
\[
    u_k \to U \qquad\text{uniformly on }
    [d_1,d_2]\times[0,d_2^{2}\iota^{-2}].
\]
\end{proof}
Before deriving the structural inequality for the reaction term,
we record a simple scaling property of the system.
For $\delta>0$, consider the modified problem
\[
    \partial_t u_{k,\delta}
      = u_{k,\delta,xx} - \delta k\,u_{k,\delta} v_{k,\delta},
    \qquad
    \partial_t v_{k,\delta}
      = - k\,u_{k,\delta} v_{k,\delta},
\]
posed with the same boundary and initial conditions as in
Lemma~\ref{lem:dirichlet}.
Defining
\[
    \tilde v_{k,\delta} := \delta v_{k,\delta},
\]
the pair $(u_{k,\delta},\tilde v_{k,\delta})$ satisfies the system
\eqref{eq:lem2.2} with the boundary-initial data $(U_0,\delta V_0)$ in place of
$(U_0,V_0)$.
Hence, by Lemma~\ref{lem:dirichlet} and the known convergence result on the half-line,
\[
    u_{k,\delta}(x,t)
    \longrightarrow
    f^{(U_0,\delta V_0)} \left(\frac{x}{\sqrt{t}}\right)
    \qquad\text{as } k\to\infty.
\]
This scaling property will be used in the construction of barrier
functions below.
\begin{lemma}\label{lem:reaction-ineq}
Let $m>1$, and let $(u_{k,\delta},v_{k,\delta})$ be the classical solutions
of the scaled system introduced above.
Fix constants $\lambda_1,\lambda_2>0$.
Then there exists $K>0$ such that, for all $k>K$,
\begin{equation}
    k u_{k,\delta}^{\,m} v_{k,\delta}
    - \lambda_1 k u_{k,\delta} v_{k,\delta}
    - \lambda_2 u_{k,\delta}
    \le 0
\end{equation}
for every $(x,t)\in(d_1,d_2)\times(0,d_2^{2}(\iota(U_0, \delta V_0))^{-2})$.
\end{lemma}
\begin{proof}
The equation for $v_{k,\delta} $ can be integrated explicitly, yielding
\[
    v_{k,\delta} (x,t) 
    = V_0 e^{-k z_{k,\delta} (x,t)},
\qquad
    z_{k,\delta} (x,t):=\int_0^t u_{k,\delta} (x,s)\,ds .
\]

We first consider the region where $u_{k,\delta}^{\,m-1}(x,t)\le \lambda_1$.
In this case $u_{k,\delta}^m \le \lambda_1 u_{k,\delta} $, and thus
\[
    k u_{k,\delta}^m v_{k,\delta}  - \lambda_1 k u_{k,\delta}  v_{k,\delta}  - \lambda_2 u_{k,\delta} 
       \le -\lambda_2 u_{k,\delta}  \le 0.
\]
Thus the desired inequality holds on this region for every $k>0$.

Next we consider points where $u_{k,\delta}^{m-1}(x,t)>\lambda_1$.
Set
\[
    A_k := \{ (x,t)\in [d_1,d_2]\times[0,d_2^2\iota^{-2}] \mid u_{k,\delta}^{m-1}(x,t)>\lambda_1 \}.
\]
By assumption $u_{k,\delta} \to U:=f^{(U_0,\delta V_0)} \left(\frac{x}{\sqrt{t}}\right)$ uniformly on $[d_1,d_2]\times[0,d_2^2\iota^{-2}]$ as $k \to \infty$.
We claim that there exists $K_0>0$ such that
\begin{equation}
    A_k \subset A_\infty:=
    \Bigl\{(x,t)\in[d_1,d_2]\times[0,d_2^2\iota^{-2}]
          \ \Big|\ U^{m-1}(x,t)\ge \tfrac{\lambda_1}{2}\Bigr\}
\end{equation}
for all $k\ge K_0$.

From the definition of $U$, there exists a constant $\eta>0$ such that
\[
    \int_0^t U(x,s)\,ds \ge \eta
    \qquad\text{for all } (x,t)\in A_\infty.
\]
Since $z_{k,\delta}\to z_\infty := \int_0^t U(x,s)\,ds$ uniformly,
we may choose $K_1>K_0$ such that for all $k\ge K_1$,
\[
    z_{k,\delta}(x,t)\ge \frac{\eta}{2}
    \qquad \text{for all } (x,t)\in A_k.
\]

Using $0\le u_{k,\delta}\le U_0$ and $v_{k,\delta} = V_0 e^{-k z_{k,\delta}}$, we estimate on $A_k$
\[
\begin{aligned}
    k u_{k,\delta}^{m} v_{k,\delta} - \lambda_1 k u_{k,\delta} v_{k,\delta} - \lambda_2 u_{k,\delta}
    &\le u_{k,\delta} \bigl( k u_{k,\delta}^{m-1} v_{k,\delta} - \lambda_2 \bigr) \\
    &\le u_{k,\delta} \bigl( U_0^{m-1} V_0 \, k e^{-k z_{k,\delta}(x,t)} - \lambda_2 \bigr) \\
    &\le u_{k,\delta} \bigl( U_0^{m-1} V_0 \, k e^{-k\eta/2} - \lambda_2 \bigr).
\end{aligned}
\]
Since $k e^{-k\eta/2}\to 0$ as $k\to\infty$, 
we can pick $K\ge K_1$ such that
\[
    U_0^{m-1} V_0 \, k e^{-k\eta/2} \le \lambda_2
    \qquad\text{for all } k\ge K.
\]
Consequently, for all $k\ge K$ the desired inequality holds throughout $A_k$.
Combining both cases completes the proof.
\end{proof}

Next we use the one-dimensional function constructed in Lemma \ref{lem:dirichlet}
to build a subsolution in the $n$-dimensional setting.
For later use, we introduce the notation of an annulus.  
For $0<\underline r<\overline r$ and $x_*\in\Omega$, let
\[
    An(\underline r,\overline r,x_*)
    := B_{\overline r}(x_*) \setminus B_{\underline r}(x_*).
\]
In Section~2 we only consider annuli centered at the origin, and we simply write
\[
    An(\underline r,\overline r) := An(\underline r,\overline r,0).
\]

\begin{lemma}\label{lem:sub-nD}
Let \(m>1\), \(0<d_1<d_3<d_2\), \(U_0,V_0>0\), and
\(0<t_0<T_0\) be fixed. Assume that
\[
    T_0<\left(\frac{d_2}{d_3}\right)^2 t_0 .
\]
Then there exist nonnegative functions \(u_k\), \(v_k\), and a constant
\(K_1>0\) such that, for every \(k\ge K_1\), the pair \((u_k,v_k)\)
satisfies
\begin{align}
    \partial_t u_k
        &\le \Delta u_k-k u_k^m v_k
        &&\text{in } An(d_1,d_2)\times(0,T_0), \label{sub u}\\
    \partial_t v_k
        &\ge -ku_kv_k
        &&\text{in } An(d_1,d_2)\times(0,T_0), \label{sub v}
\end{align}
together with
\begin{align}\label{sub bc}
    u_k &\le U_0
        &&\text{on } \partial B_{d_1}\times(0,T_0),\\
    u_k &=0
        &&\text{on } \partial B_{d_2}\times(0,T_0),
\end{align}
and
\begin{align}\label{sub ic}
    u_k(\cdot,0)=0,\qquad
    v_k(\cdot,0)=V_0
        \quad\text{in } An(d_1,d_2).
\end{align}
Moreover, there exists \(\rho>0\) such that
\[
    u_k(x,t)\ge \rho
    \qquad
    \text{for all }(x,t)\in An(d_1,d_3)\times[t_0,T_0]
\]
and all \(k\ge K_1\).
\end{lemma}

\begin{proof}
Set
\[
    \sigma:=1+\frac{n-1}{d_1},
    \qquad
    U_*:=e^{-\sigma(d_2-d_1)}U_0 .
\]
Since
\[
    T_0<\left(\frac{d_2}{d_3}\right)^2t_0,
\]
we choose \(\iota_*>0\) such that
\[
    \frac{d_3}{\sqrt{t_0}}<\iota_*<\frac{d_2}{\sqrt{T_0}},
\]
and then choose \(\delta>0\) so that
\[
    \iota(U_*,\delta V_0)=\iota_*.
\]
Set \(S_*:=d_2^2\iota_*^{-2}\), so that \(T_0<S_*\).

Let \((u_{k,\delta},v_{k,\delta})\) be the one-dimensional scaled solution
on \((d_1,d_2)\times(0,S_*)\) with data corresponding to \((U_*,V_0)\).
Equivalently, \((u_{k,\delta},\delta v_{k,\delta})\) satisfies
\eqref{eq:lem2.2} with data \((U_*,\delta V_0)\).  By
Lemma~\ref{lem:dirichlet},
\[
    u_{k,\delta}(r,t)\to
    f^{(U_*,\delta V_0)}\left(\frac r{\sqrt t}\right)
\]
uniformly on \([d_1,d_2]\times[0,S_*]\).  Moreover,
\[
    0\le u_{k,\delta}\le U_*,
    \qquad
    \partial_r u_{k,\delta}\le0,
    \qquad
    \partial_tu_{k,\delta}\ge0.
\]

For \(r=|x|\), define on \(An(d_1,d_2)\times(0,T_0)\)
\[
    u_k(x,t):=e^{\sigma(d_2-r)}u_{k,\delta}(r,t),
    \qquad
    v_k(x,t):=v_{k,\delta}(r,t).
\]
Then \(u_k,v_k\) are nonnegative and radial.  The initial conditions are
immediate.  Also, \(u_k=0\) on \(r=d_2\), while on \(r=d_1\),
\[
    u_k
    =
    e^{\sigma(d_2-d_1)}u_{k,\delta}(d_1,t)
    \le
    e^{\sigma(d_2-d_1)}U_*
    =
    U_0.
\]
Thus \eqref{sub bc} and \eqref{sub ic} hold.

The inequality for \(v_k\) follows directly from
\begin{align*}
    \partial_t v_k+ku_kv_k
    &=
    \partial_t v_{k,\delta}
    +
    k e^{\sigma(d_2-r)}u_{k,\delta}v_{k,\delta} \\
    &=
    k u_{k,\delta}v_{k,\delta}
    \left(e^{\sigma(d_2-r)}-1\right)
    \ge0,
\end{align*}
because \(r\le d_2\).  Hence \eqref{sub v} holds.

For \(u_k\), using the radial Laplacian formula and the equation for
\(u_{k,\delta}\), we obtain
\begin{align*}
&e^{-\sigma(d_2-r)}
\left(
    \partial_tu_k-\Delta u_k+ku_k^mv_k
\right)\\
={}&
-\delta k u_{k,\delta}v_{k,\delta}
-
\left(-2\sigma+\frac{n-1}{r}\right)\partial_r u_{k,\delta}
-
\sigma\left(\sigma-\frac{n-1}{r}\right)u_{k,\delta}  \\
&\quad
+
e^{(m-1)\sigma(d_2-r)}
k u_{k,\delta}^m v_{k,\delta}.
\end{align*}
Since \(r\ge d_1\), \(\partial_ru_{k,\delta}\le0\), and
\(\sigma=1+(n-1)/d_1\), the middle two terms are bounded above by
\(-u_{k,\delta}\).  Hence
\begin{align*}
&e^{-\sigma(d_2-r)}
\left(
    \partial_tu_k-\Delta u_k+ku_k^mv_k
\right)\\
\le{}&
e^{(m-1)\sigma(d_2-r)}
\Big[
    k u_{k,\delta}^m v_{k,\delta}
    -
    e^{-(m-1)\sigma(d_2-r)}
    \bigl(\delta k u_{k,\delta}v_{k,\delta}+u_{k,\delta}\bigr)
\Big].
\end{align*}
Since
\[
    e^{-(m-1)\sigma(d_2-r)}
    \ge e^{-(m-1)\sigma(d_2-d_1)}
    =:c_*>0,
\]
the bracket is at most
\[
    k u_{k,\delta}^m v_{k,\delta}
    -
    \delta c_* k u_{k,\delta}v_{k,\delta}
    -
    c_*u_{k,\delta}.
\]
By Lemma~\ref{lem:reaction-ineq}, with
\[
    \lambda_1=\delta c_*,
    \qquad
    \lambda_2=c_*,
\]
this expression is nonpositive on \((d_1,d_2)\times(0,S_*)\) for all
sufficiently large \(k\), and hence also on \((d_1,d_2)\times(0,T_0)\).
Therefore \eqref{sub u} follows.

It remains to prove positivity.  If
\((r,t)\in[d_1,d_3]\times[t_0,T_0]\), then
\[
    \frac r{\sqrt t}
    \le
    \frac{d_3}{\sqrt{t_0}}
    <
    \iota_*.
\]
Thus
\[
    f^{(U_*,\delta V_0)}
    \left(\frac r{\sqrt t}\right)>0
\]
on the compact set \([d_1,d_3]\times[t_0,T_0]\).  Hence, for some
\(\rho>0\),
\[
    f^{(U_*,\delta V_0)}
    \left(\frac r{\sqrt t}\right)\ge2\rho
    \qquad\text{on }[d_1,d_3]\times[t_0,T_0].
\]
By the uniform convergence, increasing \(K_1\) if necessary,
\[
    u_{k,\delta}(r,t)\ge\rho
    \qquad
    \text{on }[d_1,d_3]\times[t_0,T_0]
\]
for all \(k\ge K_1\).  Since
\(u_k=e^{\sigma(d_2-r)}u_{k,\delta}\) and \(e^{\sigma(d_2-r)}\ge1\),
the same lower bound holds for \(u_k\) on
\(An(d_1,d_3)\times[t_0,T_0]\).  Taking \(K_1\) larger if necessary, the
proof is complete.
\end{proof}
\section{Proof of main theorem}
We begin by introducing a notation that will be used repeatedly throughout this section.
Let $\Omega' \subset \Omega$ be a bounded domain with $C^{2}$ boundary, and let
$u_{0}\in C(\overline{\Omega'})$ satisfy
\[
    u_0(x)=0 \qquad \text{for } x\in \partial\Omega'.
\]
We denote by $h_{(\Omega',u_{0})}$ the unique classical solution of the heat equation on $\Omega'$
with Dirichlet boundary condition and initial datum $u_{0}$:
\begin{equation*}
\begin{cases}
\partial_t h_{(\Omega',u_0)}
    = \Delta h_{(\Omega',u_0)}
        &\text{in }\Omega'\times(0,T),\\
h_{(\Omega',u_0)}(x,t)=0
        &\text{on }\partial\Omega'\times(0,T),\\
h_{(\Omega',u_0)}(x,0)=u_0(x)
        &\text{in }\Omega'.
\end{cases}
\end{equation*}
The analytic semigroup theory for the Dirichlet Laplacian on $C(\overline{\Omega'})$
guarantees the existence and uniqueness of $h_{(\Omega',u_0)}$. Moreover, by standard
parabolic regularity theory,
\[
    h_{(\Omega',u_0)}
        \in C(\overline{\Omega'} \times [0,T])
        \cap C^\infty(\Omega' \times (0,T)).
\]

Before turning to the proof of the main theorem, we establish the following auxiliary lemma.
\begin{lemma}\label{lem:vk-small}
Let \((u_0,v_0)\) satisfy the assumptions of either
Theorem~\ref{thm:Dirichlet} or Theorem~\ref{thm:Neumann}, and let
\((u_k,v_k)\) be the corresponding classical solutions.
Let \(\Omega'\Subset\Omega\) and \(0<t_1<T\) be arbitrary.
Then there exist constants \(K>0\), \(\rho>0\), and a time
\(s\in[0,t_1)\) such that, for all \(k\ge K\),
\[
    u_k(x,t)\ge \rho
    \qquad
    \text{for all }(x,t)\in\Omega'\times[s,t_1],
\]
and
\[
    v_k(x,t)<k^{-3}
    \qquad
    \text{for all }(x,t)\in\Omega'\times[t_1,T).
\]
\end{lemma}
\begin{proof}
By the assumptions on \(u_0\) and \(v_0\), we can choose
\(x_0\in\operatorname{supp}(u_0)\) and \(r_0>0\) such that
\[
    B_{2r_0}(x_0)\Subset\Omega,
    \qquad
    \inf_{B_{2r_0}(x_0)}u_0
        >\frac12\sup_{\Omega}u_0
        =:\overline U_0 .
\]
Since \(u_0v_0=0\), we have \(v_0=0\) in \(B_{2r_0}(x_0)\).  Hence
\(v_k\equiv0\) in \(B_{2r_0}(x_0)\times(0,T)\), and \(u_k\) satisfies the
heat equation there.

Choose \(\tilde u_0\in C(\overline{B_{2r_0}(x_0)})\) such that
\[
    0<\tilde u_0\le \overline U_0
    \quad\text{in }B_{2r_0}(x_0),\qquad
    \tilde u_0=0
    \quad\text{on }\partial B_{2r_0}(x_0).
\]
By the comparison principle,
\[
    h_{(B_{2r_0}(x_0),\tilde u_0)}(x,t)\le u_k(x,t)
    \qquad
    \text{in }B_{2r_0}(x_0)\times[0,T).
\]
Since \(B_{r_0}(x_0)\times[0,t_1]\) is compact and
\(h_{(B_{2r_0}(x_0),\tilde u_0)}>0\) there, there exists \(\rho_0>0\) such
that
\[
    u_k(x,t)\ge \rho_0
    \qquad
    \text{for }(x,t)\in B_{r_0}(x_0)\times[0,t_1].
\]

We formally set
\[
    An(0,r_0,x_0):=B_{r_0}(x_0).
\]
Choose finitely many annuli
\[
    \{An(\underline r_i,\overline r_i,x_i)\}_{i=0}^{N}\Subset\Omega
\]
such that
\[
    \Omega'\Subset\bigcup_{i=0}^{N}An(\underline r_i,\overline r_i,x_i),
\]
and
\[
    B_{\underline r_i}(x_i)
    \Subset
    An(\underline r_{i-1},\overline r_{i-1},x_{i-1})
    \qquad (1\le i\le N).
\]
For each \(1\le i\le N\), choose
\(\hat r_i>\overline r_i\) so that
\[
    An(\underline r_i,\overline r_i,x_i)
    \subset
    An(\underline r_i,\hat r_i,x_i)
    \Subset\Omega.
\]

We next choose times
\[
    0=s_0<s_1<\cdots<s_N<t_1
\]
so that, for each \(0\le j<N\),
\[
    t_1-s_j
    <
    \left(\frac{\hat r_{j+1}}{\overline r_{j+1}}\right)^2
    (s_{j+1}-s_j).
\]

We prove by induction that, for each \(0\le j\le N\), there exist
\(\rho_j>0\) and \(K_j>0\) such that, for all \(k\ge K_j\),
\[
    u_k(x,t)\ge\rho_j
\]
for all
\[
    (x,t)\in
    \left(\bigcup_{i=0}^{j}An(\underline r_i,\overline r_i,x_i)\right)
    \times[s_j,t_1].
\]
The case \(j=0\) follows from the lower bound on \(B_{r_0}(x_0)\).

Assume the claim holds for some \(j<N\).  We apply
Lemma~\ref{lem:sub-nD} to the annulus
\[
    An(\underline r_{j+1},\hat r_{j+1},x_{j+1})
\]
with
\[
    d_1=\underline r_{j+1},
    \qquad
    d_2=\hat r_{j+1},
    \qquad
    d_3=\overline r_{j+1},
\]
and with
\[
    U_0=\rho_j,
    \qquad
    V_0=\|v_0\|_{L^\infty(\Omega)},
\]
\[
    t_0=s_{j+1}-s_j,
    \qquad
    T_0=t_1-s_j.
\]
The choice of \(s_j,s_{j+1}\) gives
\[
    T_0<
    \left(\frac{d_2}{d_3}\right)^2t_0,
\]
so the short-time annulus lemma applies.

Let \((\underline u_k,\underline v_k)\) be the corresponding subsolution,
and define the time-shifted functions
\[
    \underline u_k^{(j)}(x,t)
        :=\underline u_k(x,t-s_j),
    \qquad
    \underline v_k^{(j)}(x,t)
        :=\underline v_k(x,t-s_j)
\]
for \(t\in[s_j,t_1]\).  Then
\((\underline u_k^{(j)},\underline v_k^{(j)})\) satisfies the required
subsolution inequalities on
\[
    An(\underline r_{j+1},\hat r_{j+1},x_{j+1})
    \times(s_j,t_1).
\]
At \(t=s_j\),
\[
    \underline u_k^{(j)}=0\le u_k,
    \qquad
    \underline v_k^{(j)}=\|v_0\|_{L^\infty(\Omega)}\ge v_k.
\]
On the inner boundary
\[
    \partial B_{\underline r_{j+1}}(x_{j+1})\times[s_j,t_1],
\]
the induction hypothesis gives
\[
    u_k\ge \rho_j\ge \underline u_k^{(j)}.
\]
On the outer boundary
\[
    \partial B_{\hat r_{j+1}}(x_{j+1})\times[s_j,t_1],
\]
we have
\[
    \underline u_k^{(j)}=0\le u_k.
\]
Therefore, by the comparison principle,
\[
    \underline u_k^{(j)}\le u_k
\]
in
\[
    An(\underline r_{j+1},\hat r_{j+1},x_{j+1})
    \times[s_j,t_1].
\]
By Lemma~\ref{lem:sub-nD}, there exist \(\rho_{j+1}>0\) and
\(K_{j+1}>0\) such that, for all \(k\ge K_{j+1}\),
\[
    \underline u_k^{(j)}(x,t)\ge\rho_{j+1}
\]
for
\[
    (x,t)\in
    An(\underline r_{j+1},\overline r_{j+1},x_{j+1})
    \times[s_{j+1},t_1].
\]
Hence
\[
    u_k(x,t)\ge\rho_{j+1}
\]
on the same set.  Combining this with the induction hypothesis, and replacing
\(\rho_{j+1}\) by a smaller number if necessary, we get
\[
    u_k(x,t)\ge\rho_{j+1}
\]
on
\[
    \left(\bigcup_{i=0}^{j+1}An(\underline r_i,\overline r_i,x_i)\right)
    \times[s_{j+1},t_1].
\]
This completes the induction.

Taking \(j=N\), we obtain constants \(\rho>0\), \(K_*>0\), and
\(s:=s_N<t_1\) such that
\[
    u_k(x,t)\ge\rho
    \qquad
    \text{for }(x,t)\in\Omega'\times[s,t_1]
\]
for all \(k\ge K_*\).  Since
\[
    v_k(x,t)
    =
    v_0(x)\exp\left(-k\int_0^t u_k(x,\tau)\,d\tau\right),
\]
we have, for \(x\in\Omega'\),
\[
    v_k(x,t_1)
    \le
    \|v_0\|_{L^\infty(\Omega)}
    \exp\{-k\rho(t_1-s)\}.
\]
Thus, for sufficiently large \(k\),
\[
    v_k(x,t_1)<k^{-3}
    \qquad
    \text{for all }x\in\Omega'.
\]
Finally, since \(\partial_t v_k=-ku_kv_k\le0\), it follows that
\[
    v_k(x,t)<k^{-3}
    \qquad
    \text{for all }(x,t)\in\Omega'\times[t_1,T).
\]
The proof is complete.
\end{proof}
With Lemma~\ref{lem:vk-small} now established, we are in a position to prove
Theorem~\ref{thm:Dirichlet}.
The key point is that $v_k$ becomes sufficiently small on compact subsets
away from $t=0$, so that the reaction term $k u_k^m v_k$ no longer affects
the dynamics of $u_k$ in the limit.
\begin{proof}
[Proof of Theorem~\ref{thm:Dirichlet}]
We first show the convergence of $v_k$ away from $t=0$.
Let $\tau\in(0,T)$ and $\Omega_0\Subset\Omega$ be arbitrary.
Then Lemma~\ref{lem:vk-small} directly implies that
\[
    v_k \to 0 \qquad\text{uniformly in }
    \overline{\Omega_0}\times[\tau,T].
\]
Since $\tau$ and $\Omega_0$ were arbitrary, this proves the second
assertion of the theorem.

It remains to prove that $u_k$ converges uniformly to $h_{(\Omega,u_0)}$ on
$\overline{Q_T}$.
Let $\varepsilon>0$ be arbitrary.
Since $h_{(\Omega,u_0)}$ is continuous on $\overline{Q_T}$ and satisfies
$h_{(\Omega,u_0)}=0$ on $S_T$, we can choose a bounded domain
$\Omega'\Subset\Omega$ with $C^2$ boundary such that
\begin{align}\label{eq: m1}
	    h_{(\Omega,u_0)}(x,t) < \frac{\e}{8}
    \qquad\text{for all } (x,t)\in (\Omega\setminus\Omega')\times[0,T].
\end{align}

Next, we approximate the initial data by a smooth function.
Since $u_0\in C(\overline{\Omega})$, we may choose 
$\tilde{u}_0\in C^\infty(\Omega)$ such that
\[
    0 \le u_0(x)-\tilde{u}_0(x) < \frac{\e}{4}
    \qquad\text{for all } x\in\Omega,
\]
and such that $\partial\supp(\tilde{u}_0)$ is smooth and $\supp(\tilde{u}_0)\Subset\Omega'$.
By continuity of the Dirichlet heat flow, there exists 
$0<\tilde{t}<T$ such that
\begin{align}\label{eq: m2}
    0 \le 
    h_{(\Omega,u_0)}(x,t)
    - h_{(\supp(\tilde{u}_0),\tilde{u}_0)}(x,t)
    < \frac{\varepsilon}{2}
\end{align}
for $(x,t)\in \supp(\tilde{u}_0)\times[0,\tilde{t}]$,
and
\[
    0 \le h_{(\Omega,u_0)}(x,t) < \frac{\varepsilon}{2}
    \qquad\text{for } (x,t)\in 
    (\Omega\setminus\supp(\tilde{u}_0))\times[0,\tilde{t}].
\]
By Theorem~\ref{theo:comparison principle} and the fact that $v_0=0$ on $\supp(u_0)$, we have
\[
    h_{(\supp(\tilde{u}_0),\tilde{u}_0)}(x,t)
    \le
    u_k(x,t)
    \le
    h_{(\Omega,u_0)}(x,t)
    \qquad\text{for } (x,t)\in \supp(\tilde{u}_0)\times[0,\tilde{t}].
\]
Since $u_k\ge0$, combining the above inequalities yields
\begin{align} \label{eq:heat1}
	    |h_{(\Omega,u_0)}(x,t) - u_k(x,t)|
    < \frac{\varepsilon}{2}
    \qquad\text{for } (x,t)\in \overline{\Omega}\times[0,\tilde{t}].
\end{align}

Applying Lemma~\ref{lem:vk-small} with $\Omega'$ and $\tilde{t}$ in place of $\Omega_0$
and $\tau$, we see that, for all sufficiently large $k$,
\[
    v_k(x,t) < k^{-3}
    \qquad\text{for all } (x,t)\in \Omega'\times[\tilde{t},T).
\]
We now define 
\[
    \underline{u}(x)
    :=
    \begin{cases}
        h_{(\supp(\tilde{u}_0),\tilde{u}_0)}(x,\tilde{t}),
            & x\in\supp(\tilde{u}_0),\\[1mm]
        0, & x\in\Omega\setminus\supp(\tilde{u}_0),
    \end{cases}
\]
and introduce the function
\[
    \underline{h}(x,t)
    :=
    h_{(\Omega',\underline{u})}(x,t-\tilde{t}),
    \qquad t\ge \tilde{t}.
\]
On $\Omega'\times[\tilde{t},T)$ we also define
\[
    \underline{U}_k(x,t):=e^{-k^{-1} t}\,\underline{h}(x,t),
    \qquad
    \underline{V}_k(x,t):=k^{-3}.
\]
A direct computation shows that
\begin{align*}
    \partial_t \underline{U}_k
    - \Delta \underline{U}_k
    + k\,\underline{U}_k^{m}\underline{V}_k
    &= -k^{-1}\underline{U}_k
       + e^{-k^{-1} t}\,\partial_t\underline{h}
       - e^{-k^{-1} t}\,\Delta\underline{h}
       + k\,\underline{U}_k^{m}\underline{V}_k \\
    &= -k^{-1}\underline{U}_k
       + k^{-2}\underline{U}_k^{m} \\
    &= k^{-1}\underline{U}_k\bigl(-1 + k^{-1}\underline{U}_k^{m-1}\bigr).
\end{align*}
Since $\underline{U}_k$ is bounded on $\Omega'\times[\tilde{t},T)$, we can
choose $k$ sufficiently large so that
\[
    -1 + k^{-1}\underline{U}_k^{m-1} \le 0
    \quad\text{everywhere in }\Omega'\times[\tilde{t},T),
\]
and hence
\[
    \partial_t \underline{U}_k
    - \Delta \underline{U}_k
    + k\,\underline{U}_k^{m}\underline{V}_k
    \le 0
    \qquad\text{in } \Omega'\times[\tilde{t},T).
\]
Moreover, since $\underline{V}_k$ is a constant function, we have
\[
    \partial_t \underline{V} _k
    \ge -k\underline{U}_k\underline{V}_k
    \qquad\text{in } \Omega'\times[\tilde{t},T).
\]
Using the inequalities at $t=\tilde{t}$ and on the lateral boundary
$\partial\Omega'\times[\tilde{t},T)$, and applying
Theorem~\ref{theo:comparison principle}, we obtain
\[
    e^{-k^{-1} t}\underline{h}(x,t)
    = \underline{U}_k(x,t)
    \le u_k(x,t)
    \qquad\text{for } (x,t)\in \Omega'\times[\tilde{t},T).
\]

On the other hand, by \eqref{eq: m1}, \eqref{eq: m2} and the maximum
principle, we have
\[
    |h_{(\Omega,u_0)}(x,t)-\underline{h}(x,t)|
    < \frac{\varepsilon}{2}
    \qquad\text{for } (x,t)\in \Omega'\times[\tilde{t},T).
\]
Combining these estimates yields, for all sufficiently large $k$,
\[
    u_k(x,t) - h_{(\Omega,u_0)}(x,t)
    \ge -\varepsilon
    \qquad\text{for } (x,t)\in \Omega'\times[\tilde{t},T).
\]
Since we already know from the comparison principle that
\[
    u_k(x,t) \le h_{(\Omega,u_0)}(x,t)
    \qquad\text{for } (x,t)\in \Omega\times[0,T),
\]
it follows that
\begin{equation}\label{eq:heat2}
    |u_k(x,t)-h_{(\Omega,u_0)}(x,t)|
    \le \varepsilon
    \qquad\text{for } (x,t)\in \Omega\times[\tilde{t},T).
\end{equation}

Finally, combining \eqref{eq:heat1} with \eqref{eq:heat2} we
conclude that
\[
    \sup_{(x,t)\in\overline{Q_T}}
    |u_k(x,t)-h_{(\Omega,u_0)}(x,t)| \le \varepsilon
\]
for all sufficiently large $k$.
This proves the theorem.
\end{proof}
With Theorem~\ref{thm:Dirichlet} established, we now address Theorem~\ref{thm:Neumann}.
The proof follows the same general scheme, with a slight modification in the treatment near the boundary.
\begin{proof}
[Proof of Theorem~\ref{thm:Neumann}]
Let \(u_\infty\) be the solution of \eqref{heat-Neumann}.  As in the
Dirichlet case,
\[
    u_\infty\in C(\overline{Q_T})\cap C^\infty(Q_T),
\]
and the comparison principle gives
\begin{equation}\label{eq:n1}
    u_k\le u_\infty
    \qquad\text{in }\overline{Q_T}.
\end{equation}

Let \(\varepsilon>0\) be fixed.  Choose
\(\tilde u_0\in C^\infty(\Omega)\) such that
\[
    0\le u_0-\tilde u_0<\frac{\varepsilon}{4}
    \qquad\text{in }\Omega,
\]
and such that \(\partial\supp(\tilde u_0)\setminus\partial\Omega\) is smooth.
On \(\supp(\tilde u_0)\), consider the mixed heat problem
\begin{equation*}
\begin{cases}
\partial_t h_1=\Delta h_1
    &\text{in }\supp(\tilde u_0)\times(0,T),\\
\partial_\nu h_1=0
    &\text{on }(\partial\supp(\tilde u_0)\cap\partial\Omega)\times(0,T),\\
h_1=0
    &\text{on }(\partial\supp(\tilde u_0)\setminus\partial\Omega)\times(0,T),\\
h_1(\cdot,0)=\tilde u_0
    &\text{in }\supp(\tilde u_0).
\end{cases}
\end{equation*}
We extend \(h_1\) by zero to \(\Omega\), and denote the extension by
\(\tilde h_1\).  By the comparison principle,
\begin{equation}\label{eq:n2}
    \tilde h_1\le u_k
    \qquad\text{in }\Omega\times[0,T].
\end{equation}
Since \(u_\infty\) and \(\tilde h_1\) are continuous and have initial data
\(u_0\) and \(\tilde u_0\), respectively, we may choose
\(0<\tilde t<T\) such that
\[
    |u_\infty(x,t)-\tilde h_1(x,t)|
    <\frac{\varepsilon}{2}
    \qquad\text{for }(x,t)\in\Omega\times[0,\tilde t].
\]
Combining this with \eqref{eq:n1} and \eqref{eq:n2}, we obtain
\begin{equation}\label{eq:n3}
    |u_k(x,t)-u_\infty(x,t)|
    <\frac{\varepsilon}{2}
    \qquad\text{for }(x,t)\in\Omega\times[0,\tilde t].
\end{equation}

We next prove that \(v_k\) is small on all of \(\overline{\Omega}\) after
time \(\tilde t\).  Choose \(\Omega'\Subset\Omega\) and finitely many annuli
\[
    \{An(\underline r_i,\overline r_i,x_i)\}_{i=1}^N
\]
such that
\[
    \overline{\Omega}\subset
    \Omega'\cup\bigcup_{i=1}^N
    \bigl(\Omega\cap An(\underline r_i,\overline r_i,x_i)\bigr),
    \qquad
    B_{\underline r_i}(x_i)\Subset\Omega',
\]
and
\[
    \nu(y)\cdot (y-x_i)>0
\]
for every
\[
    y\in \partial\Omega\cap An(\underline r_i,\overline r_i,x_i).
\]
For each \(i\), choose \(\hat r_i>\overline r_i\) so that the same geometric
condition remains valid on
\(\partial\Omega\cap An(\underline r_i,\hat r_i,x_i)\).

Applying Lemma~\ref{lem:vk-small} with \(t_1=\tilde t\), we find
\(\rho>0\), \(s\in[0,\tilde t)\), and \(K>0\) such that, for all \(k\ge K\),
\[
    u_k(x,t)\ge\rho
    \qquad\text{for }(x,t)\in\Omega'\times[s,\tilde t],
\]
and
\[
    v_k(x,t)<k^{-3}
    \qquad\text{for }(x,t)\in\Omega'\times[\tilde t,T).
\]

Fix \(i\in\{1,\dots,N\}\).  Since \(\hat r_i>\overline r_i\), we can choose
\(s_i\in(s,\tilde t)\) such that
\[
    \tilde t-s
    <
    \left(\frac{\hat r_i}{\overline r_i}\right)^2(s_i-s).
\]
We apply Lemma~\ref{lem:sub-nD} to the annulus
\(An(\underline r_i,\hat r_i,x_i)\) with
\[
    d_1=\underline r_i,\qquad
    d_2=\hat r_i,\qquad
    d_3=\overline r_i,
\]
\[
    U_0=\rho,\qquad
    V_0=\|v_0\|_{L^\infty(\Omega)},
\]
and
\[
    t_0=s_i-s,\qquad
    T_0=\tilde t-s.
\]
Let \((\underline u_{k,i},\underline v_{k,i})\) be the corresponding
subsolution, shifted to the interval \([s,\tilde t]\).  On the inner
boundary, the inequality \(u_k\ge\rho\) on \(\Omega'\times[s,\tilde t]\)
implies
\[
    u_k\ge \underline u_{k,i}.
\]
On the artificial outer boundary,
\(\underline u_{k,i}=0\le u_k\).  On the part of the boundary lying on
\(\partial\Omega\), the radial monotonicity of \(\underline u_{k,i}\) and
the geometric condition \(\nu(y)\cdot(y-x_i)>0\) give
\[
    \partial_\nu \underline u_{k,i}\le0=\partial_\nu u_k.
\]
Moreover, at \(t=s\),
\[
    \underline u_{k,i}=0\le u_k,
    \qquad
    \underline v_{k,i}=\|v_0\|_{L^\infty(\Omega)}\ge v_k.
\]
Therefore, by Theorem~\ref{theo:comparison principle},
\[
    \underline u_{k,i}\le u_k
\]
in
\[
    \Omega\cap An(\underline r_i,\hat r_i,x_i)\times[s,\tilde t].
\]
By Lemma~\ref{lem:sub-nD}, there exist \(\rho_i>0\) and \(K_i>0\) such that
\[
    u_k(x,t)\ge\rho_i
\]
for
\[
    (x,t)\in
    \Omega\cap An(\underline r_i,\overline r_i,x_i)\times[s_i,\tilde t],
    \qquad k\ge K_i.
\]
Hence, for \(x\in\Omega\cap An(\underline r_i,\overline r_i,x_i)\),
\[
    v_k(x,\tilde t)
    \le
    \|v_0\|_{L^\infty(\Omega)}
    \exp\{-k\rho_i(\tilde t-s_i)\}.
\]
Thus, increasing \(K_i\) if necessary,
\[
    v_k(x,\tilde t)<k^{-3}
    \qquad
    \text{for }x\in
    \Omega\cap An(\underline r_i,\overline r_i,x_i).
\]
Since \(\partial_t v_k=-ku_kv_k\le0\), the same estimate holds on this
annulus for all \(t\in[\tilde t,T)\).  Combining this with the estimate on
\(\Omega'\), and then increasing \(K\) once more, we obtain
\begin{equation}\label{eq:n-vsmall}
    v_k(x,t)<k^{-3}
    \qquad
    \text{for all }(x,t)\in\overline{\Omega}\times[\tilde t,T)
\end{equation}
for all sufficiently large \(k\).

Finally, let \(\underline h\) be the Neumann heat flow with initial value
\(\tilde h_1(\cdot,\tilde t)\) at time \(\tilde t\):
\[
\begin{cases}
\partial_t\underline h=\Delta\underline h
    &\text{in }\Omega\times(\tilde t,T),\\
\partial_\nu\underline h=0
    &\text{on }\partial\Omega\times(\tilde t,T),\\
\underline h(\cdot,\tilde t)=\tilde h_1(\cdot,\tilde t)
    &\text{in }\Omega.
\end{cases}
\]
As in the proof of Theorem~\ref{thm:Dirichlet}, the functions
\[
    \underline U_k(x,t):=e^{-k^{-1}t}\underline h(x,t),
    \qquad
    \underline V_k(x,t):=k^{-3}
\]
form a lower comparison pair on \(\Omega\times[\tilde t,T)\), for all large
\(k\), thanks to \eqref{eq:n-vsmall}.  Therefore
\[
    e^{-k^{-1}t}\underline h(x,t)\le u_k(x,t)
    \qquad
    \text{in }\Omega\times[\tilde t,T).
\]
Using the maximum principle for the Neumann heat equation and the choice of
\(\tilde u_0\), we have
\[
    |u_\infty(x,t)-\underline h(x,t)|
    <\frac{\varepsilon}{2}
    \qquad
    \text{for }(x,t)\in\Omega\times[\tilde t,T).
\]
Together with \(u_k\le u_\infty\), this gives, for all sufficiently large
\(k\),
\[
    |u_k(x,t)-u_\infty(x,t)|<\varepsilon
    \qquad
    \text{for }(x,t)\in\Omega\times[\tilde t,T).
\]
Combining this with \eqref{eq:n3}, we conclude that
\[
    \sup_{\overline{Q_T}}|u_k-u_\infty|<\varepsilon
\]
for all sufficiently large \(k\).  The proof is complete.
\end{proof}


\section*{Acknowledgments}
This work was supported in part by JSPS KAKENHI Grant Numbers JP23H00085 and JP26K17020.


\end{document}